\documentclass[11pt,a4paper]{article}

\usepackage[T1]{fontenc}
\usepackage[utf8]{inputenc}
\usepackage{lmodern, comment}
\usepackage[margin=1in]{geometry}
\usepackage{amsmath,amssymb,amsthm,mathtools}
\usepackage{enumitem}
\usepackage{microtype}
\usepackage[colorlinks=true,linkcolor=blue,citecolor=blue,urlcolor=blue]{hyperref}

\newtheorem{theorem}{Theorem}
\newtheorem{lemma}[theorem]{Lemma}
\newtheorem{corollary}[theorem]{Corollary}
\newtheorem{prop}[theorem]{Proposition}

\newtheorem{remark}[theorem]{Remark}

\newcommand{\dd}{\mathrm{d}}
\newcommand{\dc}{\mathrm{d}^{c}}
\newcommand{\Vol}{\operatorname{Vol}}
\newcommand{\Span}{\operatorname{span}}
\newcommand{\R}{\mathbb{R}}
\newcommand{\C}{\mathbb{C}}
\newcommand{\CP}{\mathbb{CP}}
\newcommand{\K}{\mathcal K}

\title{Upper bounds of the Laplace spectrum on compact Kähler manifolds}
\author{Yuan Yuan}
\date{}

\begin{document}
\maketitle
\vspace{-2em}

\begin{abstract}
On a connected compact K\"ahler manifold without boundary, 
we derive a uniform upper bound for each Laplace eigenvalue over all
K\"ahler metrics in a given K\"ahler class. This provides an affirmative
answer to a question of Yau. 
\end{abstract}

\section{Introduction}\label{sec:introduction}

The study of upper bounds for Laplace eigenvalues seeks geometric
conditions that control the spectrum as the metric varies. On a closed manifold,
rescaling the metric changes every positive eigenvalue, so a
normalization is needed before such a question can be meaningful. For
surfaces, the natural normalization is the area. Hersch's theorem
\cite{Hersch1970} identifies the round sphere as a maximizer of the first
positive eigenvalue among metrics on the two sphere with prescribed area, 
with the estimate $ \lambda_1(S^2,g)\operatorname{Area}_g(S^2)\leq 8\pi.$ 
It is a basic example of spectral control by topology under an area
constraint.

Yang and Yau \cite{YangYau1980} established such a connection by using
meromorphic functions on compact Riemann surfaces. If a surface
$\Sigma$ admits a holomorphic map of degree $\delta$ onto
$\mathbb{CP}^1$, their argument yields
$\lambda_1(\Sigma,g)\operatorname{Area}_g(\Sigma)\leq 8\pi\delta$
for every metric compatible with its complex structure. Applying
Riemann--Roch, they obtained the genus bound
$\lambda_1(\Sigma,g)\operatorname{Area}_g(\Sigma) \leq 8\pi(\gamma+1),$
where $\gamma$ is the genus. Li and Yau \cite{LY82} subsequently
introduced conformal volume and its two-dimensional counterpart,
conformal area. Their work placed these eigenvalue estimates in a
framework that also applies to non-orientable surfaces and relates
spectral bounds to conformal immersions and the Willmore functional.
It also included estimates for K\"ahler manifolds admitting meromorphic
maps onto the projective line.

The corresponding problem for higher eigenvalues was addressed by
Korevaar \cite{K93}. He proved that there exists a universal constant $C>0$ such that
$ \lambda_k(\Sigma,g)\operatorname{Area}_g(\Sigma) \leq C(\gamma+1)k $ holds 
for every closed orientable surface of genus $\gamma$, every smooth
metric $g$ and every $ k\geq1$. Thus the entire sequence of eigenvalues admits a common
upper bound with linear dependence on the index $k$. More generally,
Korevaar obtained bounds of order $k^{\frac 2m}$ for the normalized
eigenvalues $\lambda_k(M,g)\operatorname{Vol}_g(M)^{\frac 2m}$ within a fixed
conformal class on a compact Riemannian manifold of real dimension $m$. These
results concern uniform estimates across families of metrics, which
require more than the asymptotic information provided by Weyl's law
for an individual metric.

In complex dimension greater than one, a different family of metrics
arises by fixing a K\"ahler class. Bourguignon, Li, and Yau
\cite{BLY1994} proved that a full holomorphic immersion into complex
projective space supplies an upper bound for the first positive
eigenvalue of every K\"ahler metric, 
with the bound expressed
through the pairing of the pulled-back Fubini--Study class with a power
of the K\"ahler class. Since this pairing and the volume are unchanged
within a fixed class, their theorem bounds $\lambda_1$ uniformly on
every K\"ahler class of a projective manifold.

Motivated by Korevaar’s estimates for higher eigenvalues on surfaces, Yau \cite{Y} asked about analogous bounds for higher-dimensional Kähler manifolds. 
Yau's question is the central motivation for our work. Kokarev made a major advance in this direction \cite{K20b}, who established uniform upper bounds for all Laplace eigenvalues in a fixed Kähler class whenever the manifold admits a nonconstant holomorphic map to complex projective space. His theorem extends the first-eigenvalue estimates of Bourguignon, Li, and Yau to the entire spectrum and provides a higher-dimensional Kähler analogue of Korevaar’s inequalities. The bound is linear in the eigenvalue index $k$, with its coefficient controlled by the holomorphic degree of the map and the dimensions. Importantly, the map need not be an immersion or have full-dimensional image, so the theorem also applies to certain nonprojective Kähler manifolds.

Based on Yau’s original question, Kokarev formulated  two  further 
questions \cite{K20b}. The first is whether Laplace eigenvalues are uniformly bounded in every Kähler class on an arbitrary compact Kähler manifold, without assuming the existence of a holomorphic map to projective space. The second is whether one can obtain bounds with growth \(k^{\frac 1n}\), where \(n\) is the complex dimension, as suggested by Weyl’s law. These questions distinguish the existence of uniform spectral bounds from the optimal growth of such bounds in the eigenvalue index $k$.

The present paper addresses both questions. 
We provide an affirmative answer to Yau’s question for fixed Kähler classes, without additional geometric assumptions. More precisely, we establish a uniform upper bound of order \(k\) in every Kähler class on an arbitrary closed Kähler manifold, thereby resolving Kokarev’s first question in full generality.
Under the holomorphic map hypothesis of Kokarev’s theorem, we obtain a bound of order \(k^{\frac 1d}\), where \(d\) is the complex dimension of the image. In particular, on a projective manifold we obtain a bound of order \(k^{\frac 1n}\), uniformly over all Kähler metrics in any fixed real Kähler class. This answers the second question for complex projective algebraic manifolds.

\medskip

We first set up the notation before stating our results.
For a K\"ahler class $\alpha$ on a complex manifold $(M,J)$ of complex
dimension $n$, write
\[
\K_\alpha=\{\omega:\omega\text{ is a smooth K\"ahler form on }(M,J),
[\omega]=\alpha\}.
\]
For $\omega\in\K_\alpha$ and a holomorphic line bundle $L\to M$, set
\[
\deg_{\omega}(L)
:=\int_M c_1(L)\wedge\frac{\omega^{n-1}}{(n-1)!},
\quad
\Vol_{\omega}(M):=\int_M\frac{\omega^n}{n!}.
\]
Both quantities only depend on $[\alpha]$.
Fixing a nonconstant holomorphic map
$\Phi:M\to\CP^N$, we let 
$d=\max_{p\in M}\operatorname{rank}_{\mathbb C}(d\Phi_p),
 L=\Phi^*\mathcal O_{\CP^N}(1).$
Since \(M\) is compact, \(\Phi\) is proper. Remmert’s proper mapping theorem implies that \(\Phi(M)\) is a closed complex analytic variety of \(\mathbb{CP}^N\). It follows from Chow’s theorem  that \(\Phi(M)\) is a complex projective algebraic variety of complex dimension $d$.
Moreover $1\leq d\leq n$, since $\Phi$ is a nonconstant map. Define
\begin{equation*}\label{eq:c}
c_\Phi
=\frac{2\pi\deg_{\omega}(L)}{\Vol_{\omega}(M)}
=2\pi n\,
\frac{\displaystyle\int_M c_1(L)\wedge\omega^{n-1}}
     {\displaystyle\int_M\omega^n}.
\end{equation*}
Then $c_\Phi>0$ and $c_\Phi$ is independent of the choice of
$\omega\in\K_\alpha$. 
The eigenvalues of the nonnegative Laplace--Beltrami operator
are counted with multiplicities and indexed by
$0=\lambda_0(M,\omega)<\lambda_1(M,\omega)
\leq\lambda_2(M,\omega)\leq\cdots.$
With these notations, we can state our main theorems. 
\begin{theorem}\label{thm:main}
Let $(M,J)$ be a connected compact complex manifold of complex dimension
$n\geq 1$, without boundary, and let $\alpha$ be a K\"ahler class.
There is a constant $C=C(M,J,\alpha) >0$ such that 
\begin{equation*}\label{eq:main}
\sup_{\omega\in\K_\alpha}\lambda_k(M,\omega) \leq C(k+1)
\end{equation*}
holds  for all $k\geq 0.$
\end{theorem}

\begin{theorem}\label{thm:holomorphic-map}
Let $(M,J)$ be a connected compact K\"ahler manifold of complex
dimension $n\geq1$, and let $\alpha$ be a
K\"ahler class. Assume that $M$ admits a nonconstant holomorphic map
$\Phi:M\to\CP^N$. Then
\begin{align*}
\sup_{\omega\in\K_\alpha} \lambda_k(M,\omega)
&\leq72c_\Phi\left[\left(\frac{3d!}{2}\right)^{\frac 1d}+1\right]
       k^{\frac 1d},
\end{align*}
for all $k\geq1$.
\end{theorem}

The following corollary is a direct consequence. 

\begin{corollary}
Let $(M,J)$ be a connected complex projective algebraic manifold of complex
dimension $n\geq1$, and let $\alpha $ be a
K\"ahler class. Then 
there is a constant $C=C(M,J,\alpha)>0$ such that
\begin{equation*}
\sup_{\omega\in\K_\alpha}\lambda_k(M,\omega)<Ck^{\frac 1n}
\end{equation*}
for all $k\geq1$.
\end{corollary}

\section{Preparations in K\"ahler geometry}
Let $(M,J)$ be a connected compact K\"ahler manifold of complex dimension
$n\geq1$, without boundary, and let $\alpha$ be a K\"ahler class on $M$.
For $\omega\in\K_\alpha$, the associated Riemannian metric is
$g_\omega(X,Y)=\omega(X,JY)$, and its volume form is
$\dd\mu_\omega=\frac{\omega^n}{n!}$. Fix a K\"ahler form $\omega_0\in\K_\alpha$ and 
write $g_0=g_{\omega_0}$. 
It follows by Stokes theorem that the volume $\int_M \dd\mu_\omega = \int_M \frac{\alpha^n}{n!}$ is independent of the choice of $\omega$, denoted by $V$.  
Let $d_0$ be the corresponding Riemannian distance, and let
$B_0(p,r)$ be the open balls centered at $p \in M$ with radius $r$ with respect to $g_0$. 
Write $\dc=\sqrt{-1}(\bar\partial-\partial)$. Then  $\dd\dc=2\sqrt{-1}\partial\bar\partial.$

\begin{lemma}\label{lem:euclidean}
Let $T$ be a smooth positive closed $(n-1,n-1)$-form on a neighborhood
of the closure of the Euclidean ball $\overline{B_{\mathrm E}(a,R)}$ in $\C^n$.
Then
\begin{equation}\label{eq:euclidean}
\int_{B_{\mathrm E}(a,r)}T\wedge \dd\dc \left( |z-a|^2\right)
\leq\frac{r^2}{R^2}\int_{B_{\mathrm E}(a,R)}T\wedge \dd\dc \left( |z-a|^2\right).
\end{equation}
holds for $0<r<R$. 
\end{lemma}

\begin{proof}
Let 
$m(t)=\int_{\{ |z-a|^2<t^2\}}T\wedge\dd\dc \left( |z-a|^2\right).$
Since $\dd T=0$, Stokes' theorem yields
\[
m(t)=\int_{\{|z-a|^2=t^2\}}T\wedge\dc \left( |z-a|^2\right).
\]
On the sphere $|z-a|^2=t^2$ we have
$\dc\log \left( |z-a|^2\right)=t^{-2}\dc \left( |z-a|^2 \right)$. Another application of Stokes' theorem on the annulus, yields
\begin{equation*}\label{eq:monotonicity}
\frac{m(R)}{R^2}-\frac{m(r)}{r^2}
=\int_{\{r^2< |z-a|^2<R^2\}}T\wedge\dd\dc\log \left( |z-a|^2\right) \geq 0.
\end{equation*}
Thus the claim is proved.
\end{proof}

\begin{lemma}\label{lem:global}
There exists a constant $C_0=C_0(M,J,\omega_0)>0$ such that every smooth
positive closed $(n-1,n-1)$-form $T$ satisfies
\begin{equation*}\label{eq:global}
\int_{B_0(p,r)}T\wedge\omega_0
\leq C_0r^2\int_M T\wedge\omega_0
\end{equation*}
for every $p\in M$ and $ r>0$.
\end{lemma}

\begin{proof}
Choose finitely many holomorphic charts, relatively compact open subsets
$U_j$ with holomorphic coordinates $z_j$, and compact subsets $K_j\Subset U_j$ covering
$M$. Write 
$\beta_j=\dd\dc \left( |z_j|^2 \right)$. It follows from compactness of $M$ that there exists a uniform $b\geq 1$ such that
$b^{-1}\beta_j\leq\omega_0\leq b\beta_j $ holds on all $U_j.$
There are also uniform constants $R>0$, $L>0$, and $r_*>0$ such that for
$p\in K_j$,
$\overline{B_{\mathrm E}(z_j(p),R)}\subset z_j(U_j),
z_j(B_0(p,r))\subset B_{\mathrm E}(z_j(p),Lr),
 Lr<R$ hold for every $r \in (0, r_*)$. 
For $p\in K_j$ and $0<r<r_*$,
Lemma~\ref{lem:euclidean}
implies
\begin{align*}
\int_{B_0(p,r)}T\wedge\omega_0
\leq b\int_{B_{\mathrm E}(z_j(p),Lr)} \left(z_j\right)_*T\wedge\beta_j \leq b\frac{L^2r^2}{R^2}
\int_{B_{\mathrm E}(z_j(p),R)}\left(z_j\right)_* T\wedge\beta_j \leq\frac{b^2 L^2 r^2}{R^2} \int_M T\wedge\omega_0.
\end{align*}
For $r\geq r_*$, it follows that 
\[
\int_{B_0(p,r)}T\wedge\omega_0\leq  \int_M T\wedge\omega_0 \leq r_*^{-2}r^2 \int_M T\wedge\omega_0.
\]
Thus $C_0=\max\{b^2L^2 R^{-2},r_*^{-2}\}$ suffices to yield the lemma.
\end{proof}

Let $L \to M$ be a holomorphic line bundle over a closed K\"ahler manifold $M$. 
 If $h$ is a smooth Hermitian metric on a holomorphic
line bundle $L$ and $e$ is a local nowhere-vanishing holomorphic frame,
write $H=h(e,e)=|e|_h^2.$
With our convention, the curvature of the Chern connection is the
globally defined form $-\partial\bar\partial\log H$ with 
$\left[-\frac{\sqrt{-1}}{2\pi}\partial\bar\partial\log H\right]=c_1(L).$
For a real $(1,1)$-form $\eta$, we write
$\operatorname{tr}_{\omega}\eta$ for its trace with respect to
$\omega$. Namely,  if
$\omega=\sqrt{-1} g_{j\bar k}\,dz^j\wedge d\bar z^k$ and
$\eta=\sqrt{-1}\eta_{j\bar k}\,dz^j\wedge d\bar z^k$, then
$\operatorname{tr}_{\omega}\eta=g^{j\bar k}\eta_{j\bar k}$.
Assume that there exists a nonconstant holomorphic map $\Phi: M \to \mathbb{CP}^N$. 
Since $\Phi$ is not constant, then $d\Phi$ does not vanish identically on $M$. 
Choose a Hermitian metric $h_{\mathrm{FS}}$ on $\mathcal O(1)$
and, in a local holomorphic frame, write
$H_{\mathrm{FS}}=|e|_{h_{\mathrm{FS}}}^2$ with the positive curvature form given by 
$\eta_{\mathrm{FS}} =-\sqrt{-1}\partial\bar\partial\log H_{\mathrm{FS}},$
normalized by $[\eta_{\mathrm{FS}}]=2\pi c_1(\mathcal O(1))$.
The pullback $\eta=\Phi^*\eta_{\mathrm{FS}}$
is a smooth closed semipositive $(1,1)$-form representing
$2\pi c_1(L)$. At any point where $d\Phi\neq0$, $\eta \not=0$. For any K\"ahler representative $\omega$ of $\alpha$,
$\operatorname{tr}_{\omega}\eta$ is thus nonnegative everywhere and strictly
positive somewhere. It follows that
\[ 
2\pi \int_M c_1(L)\wedge\frac{\omega^{n-1}}{(n-1)!} =
\int_M\eta\wedge\frac{\omega^{n-1}}{(n-1)!}
=\int_M \left(\operatorname{tr}_{\omega} \eta \right) \,\dd\mu_\omega>0.
\]
Consequently
\[
c_\Phi = \frac{\int_M (\operatorname{tr}_{\omega}\eta)\,\dd\mu_\omega}
     {\int_M\dd\mu_\omega}
=n\frac{\int_M\eta\wedge\omega^{n-1}}{\int_M\omega^n}
>0.
\]
By Stokes' theorem,  
the value of $c_\Phi$ is independent of the representative $\omega$.

\begin{lemma}\label{lem:trace}
Let $\omega\in\K_\alpha$.
Then there is a smooth Hermitian metric $h$ on $L$ such that, in every local
holomorphic frame with $H=|e|_h^2$,
\[
\operatorname{tr}_{\omega}
   \bigl(-\sqrt{-1}\partial\bar\partial\log H\bigr)=c_\Phi.
\]
\end{lemma}
\begin{proof}
Choose any smooth Hermitian metric $h_0$. 
Write $H_0=|e|_{h_0}^2$  and it follows that
\[
\frac{\displaystyle\int_M \operatorname{tr}_{\omega}
  \bigl(-\sqrt{-1}\partial\bar\partial\log H_0\bigr) \,\dd\mu_\omega}
     {\displaystyle\int_M\dd\mu_\omega}
=\frac{\displaystyle\int_M
                 \bigl(-\sqrt{-1}\partial\bar\partial\log H_0\bigr)\wedge
                 \frac{\omega^{n-1}}{(n-1)!}}
       {\Vol_\omega(M)}
=\frac{2\pi\deg_\omega(L)}{\Vol_\omega(M)}.
\]
Thus 
\begin{equation*}\label{eq:poisson}
\Delta_\omega f=2( \operatorname{tr}_{\omega}
  \bigl(-\sqrt{-1}\partial\bar\partial\log H_0\bigr) -c_\Phi)
\end{equation*}
admits a  smooth real solution $f$ on $M$. 
It is easy to check that  $h=e^{-f}h_0$ is the desired Hermitian metric.
\end{proof}

The following lemma is a standard result in complex geometry. 
\begin{lemma}\label{lem:sections}
For 
every integer $q\geq1$,
\begin{equation*}\label{eq:sections}
\dim_{\mathbb C}H^0(M,L^q) 
\geq\frac{q^d}{d!}.
\end{equation*}
\end{lemma}

\begin{proof}
Since $\dim_{\mathbb C} \Phi(M) =d$, 
generically chosen homogeneous linear functions $\ell_0,\ldots,\ell_d$ in $\mathbb{CP}^N$ with no common
zero on $\Phi(M)$ define a linear projection
$\pi:\Phi(M) \longrightarrow\mathbb P^d$
given by 
$\pi(y) =[\ell_0(y):\cdots:\ell_d(y)].$
Then $\pi$ is finite. 
The image of $\pi$ is therefore a closed subset of $\mathbb P^d$ of dimension
$d$, and hence $\pi$ is surjective. Consequently,
$F=\pi\circ\Phi:M\to\mathbb P^d$ is also surjective.
Let $s_j=\Phi^*\ell_j\in H^0(M,L)$. For each multi-index
$\nu=(\nu_0,\ldots,\nu_d)$ with nonnegative integer entries and
$|\nu|=\nu_0+\cdots+\nu_d=q$, 
$s^\nu=s_0^{\nu_0}\cdots s_d^{\nu_d}\in H^0(M,L^q)$. 
Let $\sum_{|\nu|=q}a_\nu s^\nu=0$ on $M$. Then the homogeneous polynomial
$P(W_0,\ldots,W_d)
=\sum_{|\nu|=q}a_\nu W_0^{\nu_0}\cdots W_d^{\nu_d}$
vanishes on $F(M)=\mathbb P^d$. Thus $P$ is the zero polynomial,
and thus all coefficients $a_\nu$ vanish. This  implies that these $\binom{q+d}{d}$ sections are linearly independent. 
 Therefore
$$ \dim_{\mathbb C}H^0(M,L^q) \geq\binom{q+d}{d} =\frac{(q+1)\cdots(q+d)}{d!} \geq\frac{q^d}{d!}.$$
\end{proof}

\section{Linear upper bound on general K\"ahler manifolds}

We first a state an elementary lemma.

\begin{lemma}\label{lem:covering}
There is an integer $N=N(M,g_0)$ such that every $g_0$-ball of radius
$r>0$ is covered by at most $N$ $g_0$-balls of radius $\frac r2$.
\end{lemma}

\begin{proof}
Since $M$ is compact, 
there exist constants
$a,b,s_0>0$ such that
$as^{2n} \leq\Vol_{g_0}(B_0(x,s))\leq bs^{2n}$
holds uniformly for  $x\in M,\ 0<s\leq s_0$.
For $r\leq s_0/2$, choose a maximal set $x_1,\ldots,x_q$ in $B_0(p,r)$ such that \(d_0(x_i,x_j)\ge \frac r2\) for all $i \not= j$. 
It follows that $B_0(p,r) \subset \cup_{j=1}^q B_0(x_j, \frac r2)$.
Moreover, $B_0(x_i, \frac r4) \cap B_0(x_j, \frac r4) =\emptyset$ for all $i \not= j$ and $B_0(x_i, \frac r4) \subset B_0(p, \frac{5r}{4})$ for all $i$. 
Therefore $qa(\frac r4)^{2n}\leq b(\frac{5r}4)^{2n}$, which implies $ q\leq  5^{2n} \frac ba.$
For $r> \frac{s_0}2$, take a finite cover of $M$ by $N'$ $g_0$-balls of radius
$\frac{s_0}4$. So the union of $N'$ $g_0$-balls of radius $\frac r2$ still covers $M$. 
Therefore $N=\max\{N',  5^{2n} \frac ba\}$ is the desired upper bound in the lemma.
\end{proof}

For a K\"ahler form $\omega$, define
$T_\omega=\frac{\omega^{n-1}}{(n-1)!}.$
Recall the Dirichlet energy of a function $f$ is given by $E_\omega(f)=
\int_M|\dd f|_{g_\omega}^2 \frac{\omega^n}{n!}$. 
Following the ideas in \cite{K93, GY99, K20a, K20b, SX22}, we construct test functions on following annuli 
$A=A(p;r,R)=\{x:r\leq d_0(p,x)<R\}$ with  $2A=\{x:\frac r2 \leq d_0(p,x)<2R\}$. 
If $r=0$, the annulus is a ball.

\begin{lemma}\label{lem:cutoff}
For every such annulus $A$ there is a Lipschitz function
$f_A:M\to[0,1]$ satisfying
$f_A=1\text{ on }A, \{f_A\ne 0\}\subset 2A$
and
\begin{equation}\label{eq:cutoffenergy}
E_\omega(f_A)
\leq 8C_0\int_M\omega_0\wedge T_\omega 
\end{equation}
for some constant $C_0 >0$.
\end{lemma}

\begin{proof}
For any $r>0$, 
let \[
h(t)=\begin{cases}
0, &0\leq t\leq \frac r2,\\
\frac{2t}{r}-1, &\frac r2<t<r,\\
1, &r\leq t\leq R,\\
2-\frac tR, &R<t<2R,\\
0,&t\geq 2R
\end{cases}
\]
and let $f_A(x)=h(d_0(p,x))$ be a Lipschitz function and thus 
$f_A \in H^1(M,g_\omega)$ for each fixed smooth K\"ahler metric $\omega$. Then 
$f_A=1\text{ on }A$, $\{f_A\ne 0\}\subset 2A$. 
It follows that
\begin{align*}
E_\omega(f_A)
&=\int_M \dd f_A \wedge\dc f_A \wedge T_\omega \\
&\leq \int_M |\dd f_A |_{g_0}^2 \omega_0 \wedge T_\omega \\
&\leq\frac{4}{r^2}\int_{B_0(p,r)}\omega_0\wedge T_\omega +\frac{1}{R^2}\int_{B_0(p,2R)}\omega_0\wedge T_\omega\\
&\leq 4C_0\int_M\omega_0\wedge T_\omega+4C_0\int_M\omega_0\wedge T_\omega,
\end{align*}
where the second inequality follows from Lemma \ref{lem:global}. 
For $r=0$, set $h=1$ on $[0,R]$, $h(t)=2-\frac tR$ for $t \in [R, 2R]$ and $h(t)=0$ for $t > 2R$. Let $f_A(x)=h(d_0(p,x))$. Then by the same argument, we have 
$$E_\omega(f_A) \leq  4C_0\int_M\omega_0 \wedge T_\omega.$$
\end{proof}

Before proving Theorem~\ref{thm:main}, we recall a  crucial estimate due to Grigor'yan-Netrusov-Yau  (cf. Theorem 1.1 in \cite{GNY}). 

\begin{theorem}[Grigor'yan--Netrusov--Yau]\label{thm:gny}
Let $(X,d)$ be a metric space satisfying the following covering property: there exists a constant $N$ such that any metric ball of radius $r$ in $X$ can be covered by at most $N$ balls
of radius $\frac r2$. Let all metric balls in $X$ be precompact sets and let $\mu$  be a  non-atomic Radon measure on
$X$ satisfying \(0<\mu(X)<\infty\). Then for 
 every integer $\ell\geq 1$, there are annuli
$A_1,\ldots,A_\ell$ satisfying
\[
\mu(A_i)\geq c(N)\frac{\mu(X)}{\ell}
\]
and  $(2A_i ) \cap (2A_j )=\emptyset$ for $i \not= j$. Here $c(N)$ is a positive constant depending only on $N$.
\end{theorem}

\begin{proof}[Proof of Theorem~\ref{thm:main}]
Fix $\omega_0\in\K_\alpha$ and let $\omega\in\K_\alpha$ be an arbitrary K\"ahler form. 
Because of Lemma~\ref{lem:covering}, we can apply  Theorem~\ref{thm:gny} to $(X,d,\mu)=(M,d_0,\mu_\omega).$ More precisely, there are 
annuli
$A_1,\ldots,A_\ell$ satisfying
\begin{equation}\label{eq:annulusmass}
\mu_\omega(A_i)\geq c(N)\frac{\mu_\omega(M)}{\ell} = c(N)\frac{V}{\ell} 
\end{equation}
for all $i=1, \cdots, \ell$ 
and  $(2A_i ) \cap (2A_j )=\emptyset$ for $i \not= j$.
For each $i$, take $f_i=f_{A_i}$ as in Lemma~\ref{lem:cutoff}.
\eqref{eq:cutoffenergy} and
\eqref{eq:annulusmass} imply
\[
E_\omega(f_i) \leq 8C_0\int_M\omega_0\wedge T_\omega =    8nC_0V,
\]
and
\[
\|f_i\|_{L^2(\mu_\omega)}^2
\geq\mu_\omega(A_i)\geq c(N)\frac V\ell.
\]
Consequently,
\begin{equation}\label{eq:individualrayleigh}
\frac{E_\omega(f_i)}{\|f_i\|_{L^2(\mu_\omega)}^2}
\leq\frac{8nC_0}{c(N)}\ell.
\end{equation}
Let $F_\ell=\Span_{\R}\{f_1,\ldots,f_\ell\}\subset H^1(M,g_\omega;\R)$ be the $\ell$-dimensional real vector space. 
$\{f_1,\ldots,f_\ell\}$
have pairwise disjoint nonzero sets and positive
$L^2$ norms, so they are linearly independent and mutually
$L^2$-orthogonal. Moreover, since
each  
$\dd f_i=0$ almost
everywhere on $\{f_i=0\}$, 
it follows from (\ref{eq:individualrayleigh}), for every nonzero $f=\sum_i a_if_i\in F_\ell$ with $a_i\in\R$,  that
\begin{equation*}\label{eq:spacerayleigh}
\frac{E_\omega(f)}{\|f\|_{L^2(\mu_\omega)}^2}
=\frac{\sum_i a_i^2 E_\omega(f_i)}
{\sum_i a_i^2\|f_i\|_{L^2(\mu_\omega)}^2}
\leq\frac{8nC_0}{c(N)}\ell. 
\end{equation*}
By the min--max principle (cf. \cite{D95, C17}),
\[
\lambda_k(M,\omega)
=
\inf_{\substack{S\subset H^1(M,g_\omega)\\ \dim S=k+1}}
\ \sup_{0\ne f\in S}
\frac{E_\omega(f)}{\|f\|_{L^2(\mu_\omega)}^2} \leq  \sup_{0\ne f\in F_{k+1}}
\frac{E_\omega(f)}{\|f\|_{L^2(\mu_\omega)}^2} \leq \frac{8nC_0}{c(N)}(k+1). 
\]
The estimate is independent of $\omega\in\K_\alpha$ and thus we finish the proof of the theorem.
  \end{proof}

\section{Sharper bounds under a holomorphic map hypothesis}

The following result essentially follows from Theorem 3.1(i) in \cite{Bordoni1998} (cf. Theorems 3.3 and Theorems 5.2 in \cite{Bordoni1996} for a more general setting): a comparison theorem for Dirac and Schrödinger spectra, which builds on the Hilbertian method developed  in \cite{GM88, B94} . 

\begin{prop}[Bordoni]\label{thm:bordoni}
Let $(X,g)$ be a connected closed Riemannian manifold, and let
$(E,h)\to X$ be a smooth real vector bundle of rank $r\geq1$,
equipped with a smooth metric $h$ and a connection $D$ compatible with $h$.
Let $0\leq\mu_1\leq\mu_2\leq\cdots$
be the eigenvalues of the connection Laplacian $D^*D$, counting multiplicities, and let
$0=\lambda_0<\lambda_1\leq\lambda_2\leq\cdots$ 
be the eigenvalues of the  
Laplacian
$\Delta_g=\mathrm{d}^*\mathrm{d}$, counting 
multiplicities. Then
\begin{equation}\label{eq:bordoni}
\mu_N\geq
\frac{1}{8(r+1)^2}\,
\lambda_{\lfloor N/(r+1)\rfloor}
\end{equation}
holds  for every integer $N\geq1.$
\end{prop}

\begin{proof}
We will show precisely how this result follows from Theorem 3.1(i) in \cite{Bordoni1998}.

Write $dv_g$ for the Riemannian volume measure. 
Consider two real Hilbert spaces $\mathcal H_E=L^2(X,E)$ and $\mathcal H_0=L^2(X)$,
with inner products $(s,t)_E=\int_X h(s,t)\,dv_g$ and $(u,v)_0=\int_X uv\,dv_g.$
The metrics $g$ and $h$ induce a metric on $T^*X\otimes E$.
For a local $g$-orthonormal tangent frame $e_1,\ldots,e_m$,
where $m=\dim_{\mathbb{R}} X$, the inner product is given by
$|Ds|^2=\sum_{a=1}^{m}h(D_{e_a}s,D_{e_a}s).$
The formal adjoint $D^*:C^\infty(X,T^*X\otimes E)\rightarrow C^\infty(X,E)$
is defined using these metrics and $dv_g$.
Let $H^1(X,E)$ be the completion of $C^\infty(X,E)$ under the norm $\|s\|_{H^1(X,E)}^2=\|s\|_E^2+\int_X|Ds|^2\,dv_g,$
and define $H^1(X)$ similarly, using $\mathrm{d}$ instead of $D$.
Consider the densely defined nonnegative quadratic forms
$Q_E(s)=\int_X|Ds|^2\,dv_g$ and $Q_0(u)=\int_X|\mathrm{d}u|_g^2\,dv_g,$
with $\operatorname{Dom}(Q_E)=H^1(X,E)$
and
$ \operatorname{Dom}(Q_0)=H^1(X).$
$Q_E$ and $Q_0$ are closed.
The nonnegative self-adjoint operators associated with \(Q_E\) and \(Q_0\) are the Friedrichs extensions of \(D^*D\) and \(\Delta_g\), respectively.
It follows from the standard elliptic theory that $D^*D$ and $\Delta_g$
 have compact resolvent. In particular, their eigenvalues
have finite multiplicities. 
Define 
$ \varpi :\mathcal H_E\rightarrow\mathcal H_0$ by 
$( \varpi s)(x)=|s(x)|_h =\sqrt{h_x(s(x),s(x))}.$
It is easy to verify that $\varpi $ satisfies
 the norm preservation condition, called Fubini's property
in \cite{Bordoni1998}. 

We now verify the Kato inequality.
First assume that $s \in C^\infty(X,E)$. For $\varepsilon>0$, let
$ u_\varepsilon =\sqrt{|s|_h^2+\varepsilon^2}$ be a smooth function.
It follows from metric compatibility that 
\[
\mathrm{d}u_\varepsilon(e_a)
=\frac{h(D_{e_a}s,s)}
       {\sqrt{|s|_h^2+\varepsilon^2}}.
\]
Consequently, by the fiberwise Cauchy--Schwarz inequality,
\begin{align*}
|\mathrm{d}u_\varepsilon|_g^2
&=
\sum_{a=1}^{m}
\frac{h(D_{e_a}s,s)^2}{|s|_h^2+\varepsilon^2} \leq
\frac{|s|_h^2}{|s|_h^2+\varepsilon^2}
\sum_{a=1}^{m}|D_{e_a}s|_h^2
\leq |Ds|^2.
\end{align*}
This implies that 
$u_\varepsilon$ is bounded in $H^1(X)$ for
$0<\varepsilon\leq1$.
Since $ 0\leq u_\varepsilon-|s|_h\leq\varepsilon$, 
$u_\varepsilon\to|s|_h=\varpi s$ uniformly on $X$ and in $L^2(X)$ as
$\varepsilon\to 0^+$. 
By the standard functional analysis argument,  $\varpi s \in H^1(X)$ and 
\[
Q_0(\varpi s)
\leq\liminf_{\varepsilon\to 0^+}Q_0(u_\varepsilon)
\leq Q_E(s),
\]
Now let $s\in H^1(X,E)$ and choose smooth sections
$s_j\to s$ in $H^1(X,E)$.
By the fiberwise triangle inequality 
$ \| \varpi s- \varpi t\|_0 \leq\|s-t\|_E,$ $\varpi  s_j\to \varpi s$ in $L^2(X)$. Since $s_j\to s$ in $H^1(X,E)$, the same argument implies that 
 $\varpi s_j$ is bounded
in $H^1(X)$. 
The standard functional analysis argument yields $\varpi s \in H^1(X)$ and 
$Q_0(\varpi s) \leq Q_E(s)$, which is the desired Kato's inequality.

Applying Theorem 3.1(i) in \cite{Bordoni1998} to $T'=D^*D, T=\Delta_g$,
 it follows that 
$$\mu_N \geq(1-C_r)\cdot0+C_r\lambda_k =\frac{1}{8(r+1)^2} \lambda_{\lfloor N/(r+1)\rfloor}.$$
\end{proof}

\begin{remark}\label{CR} 
Let $X$ be a closed K\"ahler manifold. 
Let \(L\to X\) be a Hermitian line bundle with metric \(h\) and compatible connection \(\nabla\). Equip the underlying real rank 2 vector bundle \(L_{\mathbb R}\) with
$h_{\mathbb R}(s,t):=\operatorname{Re}h(s,t)$,
and define the real connection by \(\nabla^{\mathbb R}_V s:=\nabla_Vs\) for every real vector field \(V\).
It follows from the compatibility of \(\nabla\) with \(h\) that
\[
V\bigl(h(s,t)\bigr)
=h(\nabla_Vs,t)+h(s,\nabla_Vt).
\]Taking real parts yields
\[
V\bigl(h_{\mathbb R}(s,t)\bigr)
=h_{\mathbb R}(\nabla^{\mathbb R}_Vs,t)
+h_{\mathbb R}(s,\nabla^{\mathbb R}_Vt),
\]which is precisely compatibility with the real metric.
To compare the adjoints, denote the complex and real \(L^2\) inner products by \(\langle\cdot,\cdot\rangle_{\mathbb C}\) and \((\cdot,\cdot)_{\mathbb R}\), respectively. Then $(u,v)_{\mathbb R}
=\operatorname{Re}\langle u,v\rangle_{\mathbb C}$ holds both for sections and for bundle-valued one-forms.
Thus the complex formal adjoint \(\nabla^*\) satisfies
\[
\langle\nabla s,\alpha\rangle_{\mathbb C}
=\langle s,\nabla^*\alpha\rangle_{\mathbb C}
\]for all smooth sections \(s\) and smooth \(L\)-valued one-forms \(\alpha\). Taking real parts yields
\[
(\nabla^{\mathbb R}s,\alpha)_{\mathbb R}
=(s,\nabla^*\alpha)_{\mathbb R}.
\]This is exactly the identity defining the real formal adjoint. The uniqueness therefore implies
\[
(\nabla^{\mathbb R})^*\alpha=\nabla^*\alpha,
\]and consequently
\[
(\nabla^{\mathbb R})^*\nabla^{\mathbb R}s
=\nabla^*\nabla s.
\]Thus the two connection Laplacians act identically on every smooth section.
 Since a complex basis
$s_1,\ldots,s_m$ becomes the real basis
$s_1,\sqrt{-1}s_1,\ldots,s_m,\sqrt{-1}s_m$, 
 realification leaves eigenvalues unchanged and doubles
their multiplicities.
By  \eqref{eq:bordoni},  if $\mu_j^{\mathbb R}$ denotes the eigenvalue list with real multiplicities, then
\[ \lambda_k(\Delta_g) \leq72\,\mu_{3k}^{\mathbb R}(\nabla^*\nabla) \]
for all $k\geq1$. 
\end{remark}

\begin{proof}[Proof of Theorem~\ref{thm:holomorphic-map}]

Fix $k\geq1$ and set
$q=\left\lceil\left(\frac{3d!k}{2}\right)^{\frac 1d}\right\rceil.$ 
For an arbitrary $\omega\in\K_\alpha$, it follows from Lemma~\ref{lem:trace}
that there exists 
 a smooth Hermitian metric $h$ on $L$ such that its local
squared norm $H$ satisfies
$\operatorname{tr}_{\omega}(-\sqrt{-1}\partial\bar\partial\log H)=c_\Phi$.
Equip $L^q$ with the Hermitian metric $h^q$ and denote the corresponding Chern connection by $\nabla_q$.
By the Bochner--Kodaira--Nakano
identity applied to \(L^q\)-valued \(0\)-forms (cf. page 330--332, Corollary 1.3 and (1.15) in Chapter VII
of \cite{DemaillyCADG}), 
every holomorphic
section $s \in H^0(M, L^q)$ satisfies
\[
\nabla_q^*\nabla_qs
=\operatorname{tr}_{\omega}\!\left(-\sqrt{-1}\partial\bar\partial\log H^q\right)s
=q\,\operatorname{tr}_{\omega}\!\left(-\sqrt{-1}\partial\bar\partial\log H\right)s
=q\,c_\Phi s.
\]
By Lemma~\ref{lem:sections}, these sections form a complex vector
space of dimension at least $\frac{q^d}{d!}$.

Regard $L^q$ as a real metric vector bundle of rank $2$, with metric
$\operatorname{Re} \left( h^q \right)$ and the same compatible connection.
By Remark \ref{CR}, the connection Laplacian $\nabla_q^*\nabla_q$ thus has at least $2\dim_{\mathbb C}H^0(M,L^q)\geq\frac{2q^d}{d!}\geq3k$
real linearly independent eigenvectors with eigenvalue $qc_\Phi$.
Counting eigenvalues with real multiplicities, 
we have 
\[
\mu^{\mathbb R}_{3k}(\nabla_q^*\nabla_q)\leq qc_\Phi.
\]
Proposition \ref{thm:bordoni}, applied with real rank $r=2$ and 
$N=3k$, now yields
\[
\lambda_k(M,\omega)
\leq72\mu^{\mathbb R}_{3k}(\nabla_q^*\nabla_q)
\leq72c_\Phi q.
\]
Since $k\geq1$,
\[
q\leq\left(\frac{3d!}{2}\right)^{\frac 1d}k^{\frac 1d}+1
\leq\left[\left(\frac{3d!}{2}\right)^{\frac 1d}+1\right]k^{\frac 1d}.
\]
The estimate is independent
of $\omega$, and thus the theorem is proved.
\end{proof}

\section*{Acknowledgments}
The author is partially supported by Zhejiang Provincial Natural Science Foundation of China (Grant No.~LQKWL26A0201).

\noindent Yuan Yuan, yuanyuan@westlake.edu.cn, Institute for Theoretical
Sciences, Westlake University, Hangzhou 310024, Zhejiang, China.


\begin{thebibliography}{99}



\bibitem{B94} M. Bordoni, {\em Spectral estimates for Schrödinger and Dirac-type operators on Riemannian manifolds}, Math. Ann. 298 (1994), no. 4, 693–718.

\bibitem{Bordoni1996}
M.~Bordoni,
\emph{Comparaison de spectres d'opérateurs de type Schrödinger et Dirac},
Séminaire de Théorie Spectrale et Géométrie, No.~14, Année 1995--1996,
69--81, Sémin. Théor. Spectr. Géom., 14, Univ. Grenoble I,
Saint-Martin-d'Hères, 1996.


\bibitem{Bordoni1998}
M.~Bordoni,
\emph{Spectral comparison between Dirac and Schrödinger operators},
Rend. Mat. Appl. (7) 18 (1998), no.~1, 181--196.

\bibitem{BLY1994}
J.-P.~Bourguignon, P.~Li, and S.-T.~Yau,
\emph{Upper bound for the first eigenvalue of algebraic submanifolds},
Comment. Math. Helv. \textbf{69} (1994), 199--207.


\bibitem{C17}
B. Colbois, {\em The spectrum of the Laplacian: a geometric approach}, Geometric and computational spectral theory, 1–40,
Contemp. Math., 700, Centre Rech. Math. Proc., Amer. Math. Soc., Providence, RI, 2017.

\bibitem{D95}
E. B. Davies, Spectral theory and differential operators, Cambridge Studies in Advanced Mathematics, 42. Cambridge University Press, Cambridge, 1995. x+182 pp. ISBN: 0-521-47250-4

\bibitem{DemaillyCADG}
J.-P. Demailly, Complex Analytic and Differential Geometry,
online book, 2012,
\url{https://www-fourier.univ-grenoble-alpes.fr/~demailly/manuscripts/agbook.pdf}.



\bibitem{GM88} S. Gallot and D. Meyer, {\em D'un résultat hilbertien à un principe de comparaison entre spectres. Applications}, Ann. Sci. École Norm. Sup. (4) 21 (1988), no. 4, 561–591.




\bibitem{GNY}
A.~Grigor'yan, Y.~Netrusov, and S.-T.~Yau,
\emph{Eigenvalues of elliptic operators and geometric applications},
Surveys in Differential Geometry 9 (2004), 147--217.



\bibitem{GY99}
A.~Grigor'yan and S.-T.~Yau,
\emph{Decomposition of a metric space by capacitors},
Differential equations: La Pietra 1996 (Florence), 39--75,
Proc. Sympos. Pure Math., 65, Amer. Math. Soc., Providence, RI, 1999.


\bibitem{Hersch1970}
J.~Hersch,
\emph{Quatre propri\'et\'es isop\'erim\'etriques de membranes
sph\'eriques homog\`enes},
C. R. Acad. Sci. Paris S\'er. A-B \textbf{270} (1970), A1645--A1648.




\bibitem{LY82}
P.~Li and S.-T.~Yau,
\emph{A new conformal invariant and its applications to the Willmore
conjecture and the first eigenvalue of compact surfaces},
Invent. Math. \textbf{69} (1982), no.~2, 269--291.

\bibitem{K20a}
G.~Kokarev,
\emph{Conformal volume and eigenvalue problems}, Indiana Univ. Math. J.
69 (2020), no.~6, 1975--2003.


\bibitem{K20b}
G.~Kokarev,
\emph{Bounds for Laplace eigenvalues of Kähler metrics}, Adv. Math. 365 (2020), 107061, 22 pp.

\bibitem{K93}
N.~Korevaar, \emph{Upper bounds for eigenvalues of conformal metrics},
J. Differ. Geom. 37 (1993), 73--93.








\bibitem{SX22} Y. Sire, H. Xu, 
{\em On a new functional for extremal metrics of the conformal Laplacian in high dimensions}, 
Commun. Contemp. Math. 24 (2022), no. 9, Paper No. 2150096, 16 pp.


\bibitem{YangYau1980}
P.~C.~Yang and S.-T.~Yau,
\emph{Eigenvalues of the Laplacian of compact Riemann surfaces and
minimal submanifolds},
Ann. Scuola Norm. Sup. Pisa Cl. Sci. (4) \textbf{7} (1980),
no.~1, 55--63.

\bibitem{Y}
S.-T.~Yau,
\emph{An application of eigenvalue estimate to algebraic curves defined by congruence subgroups},
Math. Res. Lett. 3 (1996), 167--172.




\end{thebibliography}
\end{document}